\documentclass[reqno,a4paper,12pt]{amsart}
\usepackage{amsmath,amscd,amsfonts,amssymb}
\usepackage{mathrsfs,dsfont}

\usepackage{tikz}
\usetikzlibrary{shapes.misc}
\usetikzlibrary{calc}
\usetikzlibrary{shapes.geometric}
\usetikzlibrary{decorations.pathreplacing}

\def\R{\mathbb{R}}

\def\Z{\mathbb{Z}}

\def\gam{\gamma}
\def\Gam{\Gamma}
\def\lam{\lambda}
\def\Lam{\Lambda}
\def\1{\mathds{1}}
\def\eps{\varepsilon}

\renewcommand\le{\leqslant}
\renewcommand\ge{\geqslant}
\renewcommand\leq{\leqslant}
\renewcommand\geq{\geqslant}

\renewcommand\hat{\widehat}
\renewcommand\Re{\operatorname{Re}}

\newcommand{\ft}[1]{\widehat #1}
\newcommand\dotprod[2]{\langle #1 , #2 \rangle}
\newcommand\mes{\operatorname{mes}}

\newcommand{\supp}{\operatorname{supp}}
\newcommand{\spec}{\operatorname{spec}}
\newcommand{\cupl}{\operatornamewithlimits{\bigcup}\limits}

\theoremstyle{plain}
\newtheorem{thm}{Theorem}

\newtheorem{lem}{Lemma}[section]

\newtheorem{prop}[lem]{Proposition}

\newcommand{\thmref}[1]{Theorem~\ref{#1}}

\newcommand{\lemref}[1]{Lemma~\ref{#1}}

\newtheorem*{theorem-m}{Theorem M}

\theoremstyle{definition}

\newenvironment{enumerate-math}
{\begin{enumerate}
\addtolength{\itemsep}{5pt}
}
{\end{enumerate}}

\newenvironment{enumerate-math-abc}
{\begin{enumerate}
\addtolength{\itemsep}{5pt}
}
{\end{enumerate}}

\begin{document}

\title{Quasicrystals with discrete support and spectrum}
\author{Nir Lev}
\address{Department of Mathematics, Bar-Ilan University, Ramat-Gan 52900, Israel}
\email{levnir@math.biu.ac.il}

\author{Alexander Olevskii}
\address{School of Mathematical Sciences, Tel-Aviv University, Tel-Aviv 69978, Israel}
\email{olevskii@post.tau.ac.il}

\thanks{Both authors are partially supported by their respective Israel Science Foundation grants.}

\begin{abstract}
We  proved  recently that a measure  on $\R$, whose support and spectrum are both uniformly discrete sets, 
must have a periodic structure. Here we show that this is not the case if the support and the spectrum are just 
discrete closed sets. 
\end{abstract}

\date{May 6, 2015}

\maketitle


\section{Introduction}

       By a Fourier quasicrystal one often means an (infinite) pure point measure $\mu$, whose Fourier transform
 is also a pure point        measure (see e.g.\  \cite{cah}, \cite{dys}).

Consider a (complex) measure $\mu$ on $\R^n$ supported on a discrete  set $\Lam$:
\begin{equation}
\label{mes1}
              \mu = \sum_{\lam\in \Lam} \mu(\lam) \delta_\lam, \quad
           \mu(\lam) \neq 0.
\end{equation}
Assume that $\mu$ is  a temperate distribution, and
          that its Fourier transform 
\[
               \ft{\mu}(t) := \sum_{\lam\in \Lam} \mu(\lam) e^{-2\pi i \dotprod{\lam}{t}} 
\]
         (in the sense of distributions)
          is also a pure point measure, namely
\begin{equation}
\label{mes2}
           \ft{\mu}=  \sum_ {s\in S}  \ft{\mu} (s) \delta_s, \quad          \ft{\mu}(s) \neq 0.
\end{equation}
          The set $\Lam$ is called the support of the measure $\mu$, while $S$ is called the spectrum.

    The classical example of such a measure comes from  Poisson's summation formula. The measure there is
    the sum of unit masses over a lattice, and the spectrum is the dual lattice.
    The problem whether other measures of Poisson type may exist, was studied
    by different authors.  See, in particular,  \cite{kah}, \cite{gui}, \cite{mey0}, \cite{deBr86}, \cite{bom},
\cite{cor1}, \cite{cor2}, \cite{lag2}. In the last paper one may find a comprehensive survey and references
    up to that date.

     Notice that a new peak of interest in the subject appeared after the experimental
    discovery in the middle of 80's  of the physical quasicrystals, demonstrating that 
    an aperiodic  atomic structure may have a diffraction pattern consisting of spots.

The  ``cut-and-project'' construction, introduced by Y.\ Meyer in the   beginning of 70's \cite{mey0},
     may serve as a good model for this phenomenon, see \cite{mey2}. It provides
     many examples of measures with uniformly discrete support and dense countable spectrum.

      On the other hand, we proved recently that if both the support and the spectrum 
     of a measure  on $\R$ are  uniformly discrete sets then the measure has a periodic
     structure.

\begin{thm}[\cite{lo1, lo2}]\label{thm0}
Let $\mu$ be a measure on $\R$ satisfying \eqref{mes1} and \eqref{mes2}, and assume that
$\Lam$ and $S$ are both uniformly discrete sets.
Then $\Lam$ is contained in a finite union of translates of an  arithmetic progression.
\end{thm}

Moreover, it was proved that such  a measure can be obtained from     Poisson's summation formula
by a finite number of shifts, multiplication on exponentials,    and taking linear combinations.
       A similar result was also proved for positive measures in $\R^n$.

      The goal of the present note is to establish the sharpness of this result,
      in the sense that  the  condition of uniform discreteness cannot be relaxed much.
      More precisely, we  prove the following:

\begin{thm}\label{thm1}
There is a (non-zero) real, signed measure $\mu$ on $\R$ satisfying \eqref{mes1} and \eqref{mes2}, such that
\begin{enumerate-math}
\item
\label{thm1:a}
$\Lam$ and $S$ are both discrete closed sets;
\item
\label{thm1:b}
$\Lam$ contains only finitely many elements of any arithmetic progression.
\end{enumerate-math}
\end{thm}

The condition \ref{thm1:b} indicates  that the  measure $\mu$
is ``non-periodic'' in a strong sense. In particular,   $\mu$ cannot be obtained from
    Poisson's summation formula by the procedures mentioned above.
In Section \ref{sec:6} below  we discuss some additional properties
of the measure in our construction that illustrate its non-periodic nature.

\subsection*{Remarks}  \quad

1. It follows from \ref{thm1:b} that
$\Lam$ may not be covered by any finite union of arithmetic progressions.
We will see that in our example, this latter property is true for $S$ as well.

2. The measure  $\mu$ constructed in the proof, as well as
its Fourier transform  $\ft\mu$, are translation-bounded  measures on $\R$.


\section{Notation}
A set $\Lambda \subset \R$ is  a  discrete closed set  if it has finitely many points in every bounded interval.
The set $\Lambda$ is called {uniformly discrete} (u.d.) if
$|\lam-\lam'| \geq \delta(\Lam) > 0$ for any two distinct points
$\lam,\lam'\in\Lam$.

By a  ``measure'' on $\R$ we mean a complex, locally finite measure (usually infinite) which is
also a temperate distribution. A  measure $\mu$  is called translation-bounded if
\begin{equation}\label{2.9}
\sup_{x \in \R} \int_x^{x+1} | d\mu| < \infty.
\end{equation}

By  the ``support'' of a pure point measure $\mu$ we mean the  
countable set of the non-zero atoms of $\mu$. This should not be confused with the
notion of support in the sense of distributions, which is always a closed set.
     In the construction  below this difference will not be important, since
     both  $\mu$  and $\ft\mu$ are supported by discrete closed sets.

The Fourier transform on $\R$ will be normalized as follows:
$$\ft \varphi (t)=\int_{\R} \varphi (x) \, e^{-2\pi i t x} dx.$$
We denote by $\supp(\varphi)$ the closed support of a Schwartz function $\varphi$,
and by $\spec (\varphi)$ the closed support of its Fourier transform $\ft\varphi$.

If $\alpha$ is a temperate distribution then $\dotprod{\alpha}{\varphi}$  denotes the action of $\alpha$ 
on a Schwartz function $\varphi$.  The Fourier transform $\ft\alpha$ is defined by 
$\dotprod{\ft\alpha}{\varphi} = \dotprod{\alpha}{\ft\varphi}$.

By  a (full-rank) lattice $\Gam \subset \R^n$ we mean the image of $\Z^n$ under some
invertible linear transformation $T$. The determinant $\det (\Gam)$ is equal to $|\det (T)|$.
The dual lattice $\Gam^*$ is the set of all vectors $\gam^*$ such that $\dotprod{\gam}{\gam^*}
 \in \Z$, $\gam \in \Gam$.

\section{Interpolation in Paley-Wiener spaces}

\subsection{}

Let $\Omega$ be a bounded, measurable set in $\mathbb R$.
We denote by $PW_\Omega$ the Paley-Wiener space consisting of all functions $f \in L^2(\R)$ whose Fourier transform
vanishes a.e.\ on $\R \setminus \Omega$. Since $\Omega$ is bounded, the
elements of the space $PW_\Omega$ are entire functions of finite exponential type.

A countable set $\Lambda \subset \R$ is called an {interpolation set} for $PW_\Omega$ if for every sequence $\{c(\lambda)\} \in \ell^2(\Lambda)$ there exists at least one  $f \in PW_\Omega$ such that $f(\lambda) = c(\lambda)$, $\lambda \in \Lambda$. It is well-known that such
$\Lam$ must be a u.d.\ set, and there is a constant $K = K(\Lam, \Omega)$ such that  the solution $f$  may be chosen to satisfy
$\|f\|_{L^2(\mathbb  R)} \leq K \| \{ c(\lambda) \} \|_{\ell^2(\Lambda)}$
(the latter follows from standard results in functional analysis).

\subsection{}
We will need to interpolate by Schwartz functions with a given spectrum. 
Recall that the topology on the Schwartz space on $\mathbb R$ is determined by the family of seminorms
$$\|f\|_{m,k}:=\sup_{x\in\mathbb R}|x^m f^{(k)}(x)|\qquad (m,k\ge0).$$

\begin{lem}\label{lem3.1}
Let $\Lambda$ be an interpolation set for $PW_\Omega$ where $\Omega$ is a compact set in $\mathbb R,$ and let $\varepsilon>0$ be given. Then, for any sequence $\{c(\lambda)\},$ $\lambda\in\Lambda$, satisfying
\begin{equation}\label{3.0}
\sup_{\lambda\in\Lambda}|c(\lambda)|\cdot(1+|\lambda|)^N<\infty\qquad (N=1,2,3,\dots),
\end{equation}
one can find  a Schwartz function $f \in PW_{\Omega+[-\varepsilon,\varepsilon]}$ which
solves the interpolation problem
 $f(\lambda)=c(\lambda)$, $\lambda\in\Lambda$,
and moreover satisfies
\begin{equation}\label{3.1}
\|f\|_{m,k}\le C_{m,k} \sup_{\lambda\in\Lambda}|c(\lambda)|\cdot(1+|\lambda|)^m \quad (m,k \geq 0)
\end{equation}
for some positive constants $C_{m,k}=C_{m,k}(\Lambda,\Omega,\varepsilon)$ which do not depend on $\{c(\lambda)\}$.
\end{lem}

\begin{proof}
Choose functions $\varphi_\lambda\in PW_\Omega$ $(\lambda\in\Lambda)$
satisfying
$$\varphi_\lambda(\lambda)=1,\qquad \varphi_\lambda(\lambda')=0\quad (\lambda'\in \Lambda,\ \lambda'\ne\lambda)$$
and 
$$\sup_{\lambda\in \Lambda}\|\varphi_\lambda\|_{L^2(\mathbb  R)}<\infty.$$
Observe that this implies that
\begin{equation}\label{3.2}
M_k:=\sup_{\lambda\in\Lambda}\|\varphi_\lambda^{(k)}\|_\infty<\infty
\end{equation}
for every $k\ge0$.
Choose a Schwartz function $\Phi$ such that $\Phi(0)=1$ and $\spec(\Phi)\subset(-\varepsilon,\varepsilon).$ Since $\Lambda$ is a u.d.\ set we have that
\begin{equation}\label{3.3}
L_{m,j}:=\sup_{x\in\mathbb R} |x|^m\sum_{\lambda\in\Lambda}(1+|\lambda|)^{-m}|\Phi^{(j)}(x-\lambda)| <\infty
\end{equation}
for every $m,j\ge0.$ Using \eqref{3.0}, \eqref{3.2} and \eqref{3.3}, we see that the function
$$f(x)=\sum_{\lambda\in\Lambda}c(\lambda) \Phi(x-\lambda) \varphi_\lambda(x)$$
is a Schwartz function in $PW_{\Omega+[-\varepsilon,\varepsilon]}$ and satisfies \eqref{3.1} with
$$C_{m,k}=\sum_{j=0}^k\binom{k}{j}L_{m,j}M_{k-j}.$$
Clearly  $f$ solves the interpolation problem, so this proves the lemma.
\end{proof}


 \section{The projection method}

\subsection{} 
 Let $\Gamma$ be a lattice in $\mathbb R^2.$ Consider the projections $p_1(x,y)=x$ and $p_2(x,y)=y,$ and assume that the restrictions of $p_1$ and $p_2$ to $\Gamma$ are injective, and so their images are dense in $\mathbb R.$
Let $\Gamma^*$ be the dual lattice, then the restrictions of $p_1$ and $p_2$ to $\Gamma^*$ are also injective and have dense images.

If $I$ is  a bounded interval in $\R$, then the set
 \begin{equation}\label{1.5}
\Lambda(\Gamma,I):=\{p_1(\gamma):\gamma\in\Gamma,\, p_2(\gamma)\in I\}
 \end{equation}
is called a ``model set'', or a ``cut-and-project'' set. Meyer observed  \cite[p.\ 30]{mey0} (see also \cite{mey2}) 
that these sets provide  examples of non-periodic  u.d.\ sets which
support a measure $\mu$, whose Fourier transform is also a pure point measure. Such a measure may be obtained by
choosing a Schwartz function $\varphi$  with $\supp(\ft\varphi) \subset I$, and taking
 \begin{equation}\label{1.2}
 \mu=\sum_{(x,y)\in\Gamma}\hat\varphi(y)\delta_x.
 \end{equation}
The Fourier transform  of $\mu$ is then the measure
 \begin{equation}\label{1.3}
 \hat \mu=\frac{1}{\det\Gamma}\sum_{(u,v)\in\Gamma^*}\varphi(v)\delta_u.
 \end{equation}
However, $\varphi$ cannot also be supported on a bounded interval, and the
support of the pure point measure $\ft\mu$ is generally everywhere dense in $\R$.

Our approach is inspired by Meyer's construction, but an essential difference is that in our
example, neither $\varphi$ nor $\ft\varphi$ will have a bounded support. 
We will nevertheless see that by a special choice of the function $\varphi$,   the measure 
\eqref{1.2} and its Fourier transform \eqref{1.3} can each be supported by a discrete closed set,
obtained by a certain generalization of the cut-and-project construction.

\subsection{} 
For completeness of the exposition, we formulate the correspondence between the measures
\eqref{1.2} and \eqref{1.3}  for a general Schwartz function $\varphi$, and include the short proof.

 \begin{lem}\label{lem1.1}
 Let $\varphi$ be a Schwartz function on $\mathbb R.$ Then \eqref{1.2} defines
a translation-bounded measure $\mu$ on $\mathbb R,$ whose Fourier transform is the (also translation-bounded) measure \eqref{1.3}.
 \end{lem}

\begin{proof}
Fix $M>0$ such that every cube with side length $1$ contains at most $M$ points of $\Gam$.
For $x\in\R$ consider the cubes
$B_k(x) := [x, x+1] \times [k, k+1)$, $k\in\Z.$
Then
$$\int_x^{x+1} | d\mu| 
= \sum_{k\in\Z} \sum_{\gam\in \Gam\cap B_k(x)} |\ft\varphi(p_2(\gam))|
\leq M\sum_{k\in\Z} \sup_{y \in [k,k+1)} |\ft\varphi(y)| =: C(\Gam, \varphi) < \infty.$$
Hence $\mu$
 is a translation-bounded measure, and in the same way one can show that 
the measure in \eqref{1.3} is also  translation-bounded. It remains to show that
the latter measure is indeed the Fourier transform of $\mu$.

 Let $\psi$ be a Schwartz function on $\mathbb R.$ Then 
$$ \dotprod{\ft\mu}{\psi} =  \dotprod{\mu}{\ft\psi} =\sum_{(x,y)\in\Gamma}
\ft\psi(x) \ft\varphi(y) = \frac{1}{\det \Gam} \sum_{(u,v)\in\Gamma^*} \psi(u) \varphi(v),$$
where the last equality follows from Poisson's summation formula.
As this holds for every Schwartz function $\psi$, this confirms \eqref{1.3}.
\end{proof}

\subsection{}
Model sets also play an interesting role in the interpolation theory in Paley-Wiener spaces.
It was proved in \cite{ou1,ou2} that there exist ``universal'' sets $\Lam$ of positive density,
which serve as a set of interpolation for  $PW_\Omega$ 
whenever $\Omega$ is a finite union of intervals with sufficiently large measure.
An example of such universal interpolation sets can also be obtained by the ``cut-and-project'' construction:

  \begin{theorem-m}[\cite{matei-meyer-quasicrystals,matei-meyer-simple}]
Let $I$ be a bounded interval in $\mathbb R$. Then the set
$\Lambda(\Gamma,I)$ defined by \eqref{1.5} is an interpolation set for $PW_\Omega$ whenever $\Omega$ is a finite
union of intervals such that
$$\mes(\Omega)>\frac{|I|}{\det \Gamma}.$$
\end{theorem-m}
Here $|I|$ denotes the length of the interval $I$.


\section{The construction}

\subsection{}
Suppose that we are given a sequence of real numbers
$$0=a_0<a_1<a_2<a_3<\cdots,\qquad a_n\to\infty\quad (n\to\infty)$$
and also another sequence
$$0<h_1<h_2<h_3<\cdots,\qquad h_n\to\infty\quad (n\to\infty).$$
We partition the plane $\mathbb R^2$ into two disjoint sets $A,B$ defined by
 \begin{equation}\label{9.1}
A=\cupl_{n=1}^\infty\{(x,y):|x|\ge a_{n-1},\, |y|\le h_n\},
 \end{equation}
 \begin{equation}\label{9.2}
B=\cupl_{n=1}^\infty\{(x,y):|x|< a_{n},\, |y|> h_n\},
 \end{equation}
and consider the two sets
$$\Lambda:=\{p_1(\gamma):\gamma\in\Gamma\cap A\},\\[2pt]$$
$$Q:=\{p_2(\gamma):\gamma\in\Gamma\cap B\}.$$
Observe that $\Lambda$ and $Q$ are both discrete closed sets in $\mathbb R$
(see Figure \ref{fig:lattice}).
Also observe that if $\varphi$ is a Schwartz function such that $\hat\varphi$ vanishes on $Q$, then the support of the measure $\mu$ in \eqref{1.2} is contained in $\Lambda.$


\begin{figure}[htb]
\centering
\begin{tikzpicture}[scale=.75, projeksjonsprikk/.style={red!75!white}, latticefarge/.style={blue!75!white}]
\draw (-4.7, 0) -- (4.7, 0);
\draw (0, -4.7) -- (0, 4.7);

\tikzset{cross/.style={cross out, draw=black, minimum size=6*(#1-\pgflinewidth), inner sep=0pt, outer sep=0pt}, cross/.default={1pt}}
\foreach \x in {-1, ..., 1} {
  \foreach \y in {-1, ..., 1} {
		\draw (-1*\x-2*\y, - 2.5*\x +1.25*\y ) node[cross] {};
  }
}

\draw (3,4) -- (3,3) -- (2,3) -- (2,2) -- (1,2) -- (1,1) -- (-1,1) -- (-1,2) -- (-2,2) -- (-2,3) -- (-3,3) -- (-3,4);
\draw (3,-4) -- (3,-3) -- (2,-3) -- (2,-2) -- (1,-2) -- (1,-1) -- (-1,-1) -- (-1,-2) -- (-2,-2) -- (-2,-3) -- (-3,-3) -- (-3,-4);

\draw[-latex] (1-2+0.08, 2.5+1.25) -- (-0.08, 2.5+1.25);
\draw[-latex] (-1+2-0.08, -2.5-1.25) -- (0.08, -2.5-1.25);
\draw[-latex] (1-0.08, 2.5) -- (0.08, 2.5);
\draw[-latex] (-1+0.08, -2.5) -- (-0.08, -2.5);

\fill (0, 2.5+1.25) circle (0.06);
\fill (0, -2.5-1.25) circle (0.06);
\fill (0, -2.5) circle (0.06);
\fill (0, 2.5) circle (0.06);

\draw[-latex] (-1-2, -2.5+1.25+0.08) -- (-1-2, -0.08);
\draw[-latex] (1+2, 2.5-1.25-0.08) -- (1+2, 0.08);
\draw[-latex] (-2, 1.25-0.08) -- (-2, 0.08);
\draw[-latex] (2, -1.25+0.08) -- (2, -0.08);

\fill (-1-2,0) circle (0.06);
\fill (1+2,0) circle (0.06);
\fill (-2, 0) circle (0.06);
\fill (2, 0) circle (0.06);

\fill (0, 0) circle (0.06);

\fill (3,3) circle (0.06);
\draw (3,3) node[right] {$(a_n, h_n)$};

\draw (0, 4.7) node[above] {$Q$};
\draw(4.7, 0) node[right] {$\Lambda$};
\draw(-4,2) node[scale=1] {$\Gamma$}; 

\end{tikzpicture}
\caption{Construction of $\Lam$ and $Q$ from the lattice $\Gam$.\label{fig:lattice}}
\end{figure}
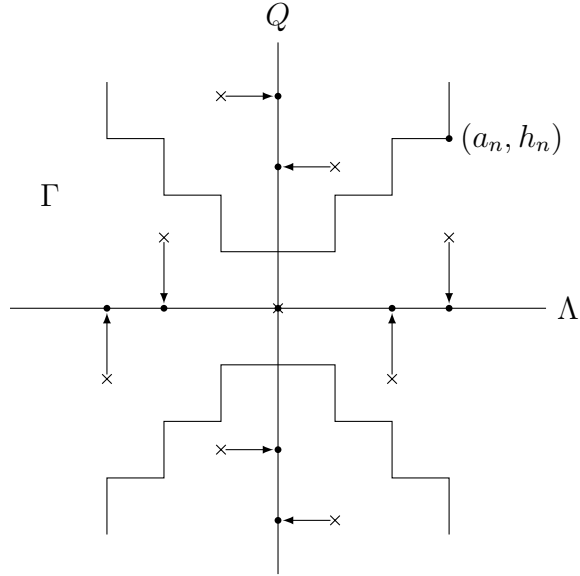


Suppose now that we are given two other sequences $\{a_n^*\}, $ $\{h_n^*\}$ with properties similar to $\{a_n\},$ $\{h_n\},$ and that these two sequences determine a partition of $\mathbb R^2$ into two disjoint sets $A^*,$ $B^*$ defined similarly to $A,B.$ Let
$$S:=\{p_1(\gamma^*):\gamma^*\in\Gamma^*\cap A^*\},$$
$$Z:=\{p_2(\gamma^*):\gamma^*\in\Gamma^*\cap B^*\}.$$
As before, $S$ and $Z$ are two discrete closed sets in $\mathbb R$
(see Figure \ref{fig:dual}).
 And, if $\varphi$ vanishes on $Z,$ then according to \eqref{1.3} the spectrum of the measure $\mu$ is contained in $S.$


\begin{figure}[htb]
\centering
\begin{tikzpicture}[scale=.75, projeksjonsprikk/.style={red!75!white}, latticefarge/.style={blue!75!white}]
\draw (-4.7, 0) -- (4.7, 0);
\draw (0, -4.7) -- (0, 4.7);

\tikzset{cross/.style={cross out, draw=black, minimum size=6*(#1-\pgflinewidth), inner sep=0pt, outer sep=0pt}, cross/.default={1pt}}
\foreach \x in {-1, ..., 1} {
  \foreach \y in {-1, ..., 1} {
		\draw (1*\x+2*\y, - 2.5*\x +1.25*\y ) node[cross] {};
  }
}

\draw (3,4) -- (3,3) -- (2,3) -- (2,2) -- (1,2) -- (1,1) -- (-1,1) -- (-1,2) -- (-2,2) -- (-2,3) -- (-3,3) -- (-3,4);
\draw (3,-4) -- (3,-3) -- (2,-3) -- (2,-2) -- (1,-2) -- (1,-1) -- (-1,-1) -- (-1,-2) -- (-2,-2) -- (-2,-3) -- (-3,-3) -- (-3,-4);

\draw[-latex] (-1+2-0.08, 2.5+1.25) -- (0.08, 2.5+1.25);
\draw[-latex] (1-2+0.08, -2.5-1.25) -- (-0.08, -2.5-1.25);
\draw[-latex] (-1+0.08, 2.5) -- (-0.08, 2.5);
\draw[-latex] (1-0.08, -2.5) -- (0.08, -2.5);

\fill (0, 2.5+1.25) circle (0.06);
\fill (0, -2.5-1.25) circle (0.06);
\fill (0, -2.5) circle (0.06);
\fill (0, 2.5) circle (0.06);

\draw[-latex] (1+2, -2.5+1.25+0.08) -- (1+2, -0.08);
\draw[-latex] (-1-2, 2.5-1.25-0.08) -- (-1-2, 0.08);
\draw[-latex] (2, 1.25-0.08) -- (2, 0.08);
\draw[-latex] (-2, -1.25+0.08) -- (-2, -0.08);

\fill (-1-2,0) circle (0.06);
\fill (1+2,0) circle (0.06);
\fill (-2, 0) circle (0.06);
\fill (2, 0) circle (0.06);

\fill (0, 0) circle (0.06);

\fill (3,3) circle (0.06);
\draw (3,3) node[right] {$(a_n^*, h_n^*)$};

\draw (0, 4.7) node[above] {$Z$};
\draw(4.7, 0) node[right] {$S$};
\draw(-4,2) node[scale=1] {$\Gamma^*$}; 

\end{tikzpicture}
\caption{A similar construction of $S$ and $Z$ from the dual lattice $\Gam^*$.\label{fig:dual}}
\end{figure}
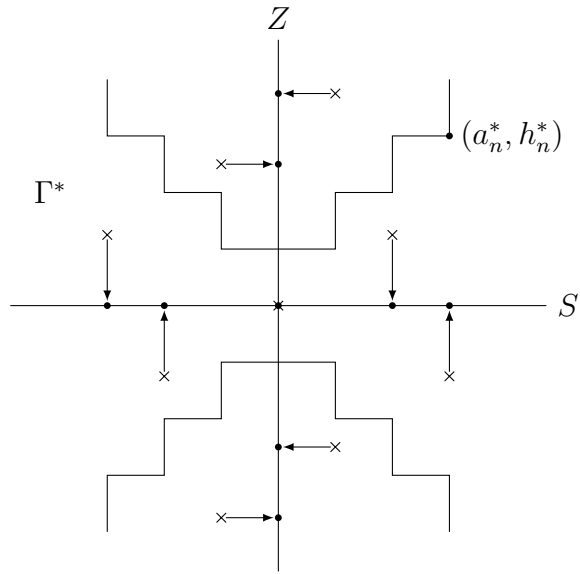


\subsection{}
Our goal will thus be  to construct sequences $\{a_n\},$ $\{h_n\}$ and $\{a_n^*\},$ $\{h_n^*\}$
 with the properties above, and a Schwartz function $\varphi$ (not identically zero),
such that $\varphi$ and $\hat \varphi$ are both real-valued, $\varphi$ vanishes on $Z$ and $\hat\varphi$ vanishes on $Q$.
As we have seen, this would give a non-zero measure $\mu$ satisfying  property \ref{thm1:a} in  \thmref{thm1}.

We  choose the sequences $\{a_n\}, $ $\{h_n\}$ in an arbitrary way, and this choice defines the sets $\Lambda$ and $Q.$

We  also choose the $\{a_n^*\}$ arbitrarily, and for each $n\geq 0$ we let $\Omega_n$ be a finite union of closed intervals in $\R$, such that
$$\Omega_n=-\Omega_n,\qquad \Omega_n\subset\mathbb R\setminus Q,\qquad \mes(\Omega_n)>\frac{2a_n^*}{\det\Gamma^*},\qquad
\Omega_0\subset \Omega_1\subset \Omega_2\subset\cdots$$
(it may be convenient here to notice that $-Q=Q).$

Let  $\varphi_0$ be a Schwartz function that is real-valued, even, 
$$\varphi_0(0)=1,\quad\spec(\varphi_0)\subset\Omega_0.$$

Now we construct by induction on $n$ the sequence $\{h_n^*\}$ and Schwartz functions $\varphi_n,$ real-valued and even, such that

\begin{enumerate-math-abc}
\item
\label{item:nonzero}
$\varphi_n(0)=1;$
\item
\label{item:spec}
$\spec(\varphi_n)\subset\Omega_n;$
\item
\label{item:approx}
$\|\varphi_n-\varphi_{n-1}\|_{m,k}<2^{-n}$ \quad $(0\le m,k\le n);$
\item
\label{item:vanish}
$\varphi_n$ vanishes on $Z_n,$ where
$$Z_n:=\{p_2(\gamma^*):\gamma^*\in\Gamma^*\cap B_n^*\},$$
$$B_n^*:=\cupl_{k=1}^n\{(x,y):|x|<a_k^*,\, |y|>h_k^*\}.$$
\end{enumerate-math-abc}

Property \ref{item:approx} ensures that $\varphi_n$ converges in the Schwartz space to a limit $\varphi$, real-valued and even, not identically zero by \ref{item:nonzero}, which vanishes on $Z$ due to \ref{item:vanish}, and such that $\hat \varphi$ vanishes on $Q$ due to \ref{item:spec}. So  the measure $\mu$ in \eqref{1.2} has support in $\Lambda$ and spectrum in $S$ as required.

To construct the number $h_n^*$ and the function $\varphi_n$ at the $n$'th step of the induction,
we let $J$ denote a finite union of closed intervals such that
$$\mes(J)>\frac{2a_n^*}{\det\Gamma^*},\qquad J+[-\varepsilon,\varepsilon]\subset\Omega_n$$
for an appropriate $\varepsilon>0$. Now consider the model set
$$X = X_n :=\{p_2(\gamma^*):\gamma^*\in\Gamma^*,\, |p_1(\gamma^*)|<a_n^*\}.$$
By Theorem~M it is an interpolation set for $PW_J.$
Define
$$C:=\sup_{0\le m,k\le n} C_{m,k}(X,J,\varepsilon)$$
where $C_{m,k}(X,J,\varepsilon)$ is the constant from \lemref{lem3.1}.  

Since $\varphi_{n-1}$ is a Schwartz function, we have
$$\sup_{\lambda\in X}|\varphi_{n-1}(\lambda)|\cdot(1+|\lambda|)^N<\infty\qquad (N=1,2,3,\dots).$$
We choose the number $h_n^*$ sufficiently large such that
$$\sup_{\lambda\in X,|\lambda|>h_n^*}|\varphi_{n-1}(\lambda)|\cdot(1+|\lambda|)^n<\frac{1}{C\cdot 2^n}\;,$$
and consider the interpolation problem
\begin{equation}\label{5.1}
f(\lambda)=c(\lambda),\qquad \lambda\in X,
\end{equation}
where
$$c(\lambda):=\begin{cases}
0,&\quad \lambda\in X,\ |\lambda|\le h_n^*,\\
\varphi_{n-1}(\lambda),&\quad \lambda\in X,\ |\lambda|>h_n^*.\end{cases}$$
By \lemref{lem3.1}, there is a Schwartz function $f$  satisfying \eqref{5.1} such that 
$$\spec(f)\subset J+[-\varepsilon,\varepsilon]\subset\Omega_n$$
and
$$\sup_{0\le m,k\le n}\|f\|_{m,k}<2^{-n}.$$
Since $\{c(\lambda)\}$ is a real-valued, even sequence, by replacing $f$ with
$$\Re\left[\frac{f(t)+f(-t)}{2}\right]$$
we may assume that $f$ is real-valued and even. We then take
$$\varphi_n:=\varphi_{n-1}-f.$$

It is clear that $\varphi_n$ satisfies conditions \ref{item:spec} and \ref{item:approx} above.
To check that \ref{item:nonzero} and \ref{item:vanish} are also satisfied, we first use the fact that $f(\lambda)=0$ for $\lambda\in X,$ $|\lambda|\le h_n^*.$ It implies that $\varphi_n(0)=\varphi_{n-1}(0)=1$ and 
$$\varphi_n(\lambda)=\varphi_{n-1}(\lambda)=0, \quad \lambda\in Z_n\cap[-h_n^*,h_n^*],$$
 where the latter  is true because $Z_n\cap[-h_n^*,h_n^*] \subset  Z_{n-1}$
 and $\varphi_{n-1}$ vanishes on $Z_{n-1}$.
 On the other hand, since $f(\lambda)=\varphi_{n-1}(\lambda)$ for $\lambda\in X,$ $|\lambda|>h_n^*,$ we obtain also
$$\varphi_n(\lambda)=0, \quad \lambda\in Z_n\setminus[-h_n^*,h_n^*].$$
 This confirms that conditions \ref{item:nonzero}--\ref{item:vanish} hold, and completes the inductive construction.


\section{Arithmetic progressions}
\label{sec:6}

\subsection{}
To complete the proof of \thmref{thm1}, it remains to show how to satisfy 
property \ref{thm1:b} in the construction above. For this we  need
the following proposition, which  provides the relation of the construction 
 to arithmetic progressions.

 \begin{lem}\label{lem1.2}
Let $P$ be an arithmetic progression in $\R$. Then the set
$$\{(x,y)\in\Gamma : x \in P\}$$
is contained in a straight line in $\R^2$ which is not parallel to the $x$-axis.
 \end{lem}

\begin{proof}
Suppose  that $\gam_1, \gam_2, \gam_3$ are three distinct points in $\Gam$, whose images under $p_1$ lie in 
$P$. Then there are non-zero integers  $m_1, m_2, m_3$ such that
$m_1+ m_2+ m_3= 0$ and
\[
m_1 p_1(\gam_1) + m_2 p_1(\gam_2) + m_3 p_1(\gam_3) =0.
\]
Since $p_1$ restricted to $\Gamma$ is injective, this implies that
\[
m_1 \gam_1 + m_2 \gam_2 + m_3 \gam_3 =0,
\]
hence the points $\gam_1, \gam_2, \gam_3$ lie on a line in $\R^2$.
Since $p_2$ restricted to $\Gamma$ is also injective, it follows that
the line is not parallel to the $x$-axis.
\end{proof}

Now recall that the sequences $\{a_n\}, $ $\{h_n\}$ have been chosen in an arbitrary way.
To satisfy  property \ref{thm1:b} we choose them such  that the domain $A$ defined
by \eqref{9.1}  contains only a bounded part of any straight line not parallel to the $x$-axis
(for this it is enough that  $a_n$ grows much faster than  $h_n$). 
By \lemref{lem1.2} this implies that
$\Lam$ contains only finitely many elements of any arithmetic progression.
So \thmref{thm1} is proved.

\subsection{}
There are several other ways to illustrate the non-periodic nature of the   measure in our example,
in addition to property \ref{thm1:b} stated above. 
Observe that \ref{thm1:b} implies:

\begin{enumerate-math}
\setcounter{enumi}{2}
\item
\label{thm1:c}
$\Lam$ may not be covered by any finite union of arithmetic progressions.
\end{enumerate-math}

\noindent
This property is true for the spectrum as well, namely:

\begin{enumerate-math}
\setcounter{enumi}{3}
\item
\label{thm1:d}
Also $S$ may not be covered by a finite union of arithmetic progressions.
\end{enumerate-math}

\noindent
In fact, the properties \ref{thm1:c} and \ref{thm1:d} hold due to the following

 \begin{prop}\label{prop1.3}
The support of any measure $\mu$ of the form \eqref{1.2} may not be covered
by a finite union of arithmetic progressions, and the same is true for the support of the measure $\hat\mu$,
unless the function $\varphi$ vanishes identically.
 \end{prop}
 
\begin{proof}
Indeed, if the support of the measure $\mu$ is contained
in a finite union of arithmetic progressions, then by \lemref{lem1.2} the set
$$\Gamma_0 :=\{(x,y)\in\Gamma : \hat\varphi(y) \neq 0\}$$
must be contained in a finite union of lines.
Thus $p_2(\Gamma_0)$ is a discrete closed set in $\R$.
But on the other hand, since $p_2(\Gamma)$ is dense in $\R$, the closure of $p_2(\Gamma_0)$ must be
equal to $\supp(\ft\varphi)$, a contradiction. 
Hence the support of $\mu$ is not contained in any finite union of arithmetic progressions,
and similarly, the same is true for the support of $\ft{\mu}$.
\end{proof}

 In the first version of this paper the property \ref{thm1:b} in \thmref{thm1} 
was not mentioned explicitly,  being replaced  by \ref{thm1:c} and \ref{thm1:d} 
above. In this weaker  form,  another  proof of our result  was given  by Kolountzakis \cite{kol},
who used an infinite  sum of  appropriately chosen  Poisson  measures.

One may actually consider stronger versions of properties \ref{thm1:c} and \ref{thm1:d}.
For instance, given $\eps > 0$ there is a decomposition of $\mu$ as the sum of two measures
$\mu = \mu_1 + \mu_2$ such that $\mu_1$ is supported on a model set and $\|\mu_2\| < \eps$,
where by  $\|\cdot\|$ we denote the natural norm on the space of translation-bounded measures
defined by \eqref{2.9}. It follows (again by  \lemref{lem1.2}) that $\mu$ may not even be approximated,
with respect to this norm, to arbitrary degree by measures whose supports are contained
in finite unions of arithmetic progressions. The same is true for the measure $\ft{\mu}$.


\section{Remarks}

\subsection{}
One may obtain results of similar type also in $\R^n$. For instance,  the
product $\mu \times \dots \times \mu$ ($n$ times) of the measure $\mu$ in our
proof gives an example of a measure in $\R^n$, whose support and spectrum
are both discrete closed sets, and the support cannot be covered by any finite
union of translated lattices, nor may it contain infinitely many elements of an
arithmetic progression.

\subsection{}
By choosing the sequence $\{a_n\}$ increasing sufficiently fast, the measure $\mu$
may be constructed with the additional property that the minimal distance between
consecutive points of $\Lam$ in the interval $(-R,R)$ approaches zero 
arbitrarily slowly as $R \to \infty$.

\subsection{}
In \cite{lo1, lo2} we obtained an affirmative answer to Problem 4.1(a) in \cite{lag2},
which asked whether it is true that a positive measure $\mu$ in $\R^n$ may have 
 uniformly discrete   support and spectrum 
only if each of them is contained in a finite union of translates of some lattice.

Problem 4.1(b) from that paper asks whether the periodic structure is still necessary
if the  support and the spectrum  are just discrete closed sets.
Here we basically answer this question in the negative, but without the positivity of the measure $\mu$.
At present we do not know whether in our construction one can get a positive measure.

\subsection{}
To each  measure satisfying \eqref{mes1} and \eqref{mes2} corresponds a weighted summation formula
\begin{equation}\label{6.1}
\sum_{\lam\in \Lam} \mu(\lam) \ft f (\lam) = \sum_ {s\in S}  \ft{\mu} (s) f(s),
\end{equation}
which holds for any Schwartz function $f$. If $\mu$ and $\ft\mu$ are  translation-bounded measures
(or, more generally, measures with polynomial growth) then both series in \eqref{6.1} converge absolutely,
otherwise an appropriate summation method should be used to sum them.

An interesting  summation formula may be found in  \cite[p.\ 265]{gui}, which involves weighted sums of
$f$ and $\ft f$ at the nodes $\{\pm (n+\tfrac1{9})^{1/2}\}$ $(n=0,1,2,\dots).$
 However it is not clear to which class of functions it applies. In particular, whether          it corresponds to  a  temperate  distribution.

Remark that the nodes in this example contain two arithmetic progressions $3 \Z \pm \frac1{3}$.



\begin{thebibliography}{99}

\bibitem[BT87]{bom}
E. Bombieri, J. E. Taylor, 
Quasicrystals, tilings, and algebraic number theory: some preliminary connections.
The legacy of Sonya Kovalevskaya, Contemp. Math., vol. 64, Amer. Math. Soc., Providence, RI, 1987, pp.  241--264.

\bibitem[CT87]{cah}
J. W. Cahn, J. E. Taylor,
An introduction to quasicrystals. 
The legacy of Sonya Kovalevskaya, Contemp. Math., vol. 64, Amer. Math. Soc., Providence, RI, 1987, pp. 265--286.

\bibitem[Cor88]{cor1}
A. C\'ordoba, 
La formule sommatoire de Poisson. 
C. R. Acad. Sci. Paris S\'er. I Math. \textbf{306} (1988), 373--376. 

\bibitem[Cor89]{cor2}
A. C\'ordoba,
Dirac combs. 
Lett. Math. Phys. \textbf{17} (1989), 191--196. 

\bibitem[deBr86]{deBr86}
N. G. de Bruijn,
Quasicrystals and their Fourier transform.
Nederl. Akad. Wetensch. Indag. Math. \textbf{48} (1986), 123--152.

\bibitem[Dys09]{dys}
F. Dyson,
Birds and frogs. 
Notices Amer. Math. Soc. \textbf{56} (2009), 212--223. 

\bibitem[Gui59]{gui}
A. P. Guinand, 
Concordance and the harmonic analysis of sequences. 
Acta Math. \textbf{101} (1959), 235--271.

\bibitem[KM58]{kah}
J.-P. Kahane, S. Mandelbrojt,
Sur l'\'equation fonctionnelle de Riemann et la formule sommatoire de Poisson.
Ann. Sci. \'Ecole Norm. Sup. \textbf{75} (1958), 57--80.

\bibitem[Kol15]{kol}
M. Kolountzakis, Fourier pairs of discrete support with little structure, preprint, \texttt{arXiv:1502.06283}.

\bibitem[Lag00]{lag2}
J. C. Lagarias, 
Mathematical quasicrystals and the problem of diffraction.  Directions in mathematical quasicrystals, 61--93, 
CRM Monogr. Ser., 13, Amer. Math. Soc., Providence, 2000. 

\bibitem[LO13]{lo1}
N. Lev, A. Olevskii, 
Measures with uniformly discrete support and spectrum.
C. R. Math. Acad. Sci. Paris \textbf{351} (2013), 599-603.

\bibitem[LO14]{lo2}
N. Lev, A. Olevskii, 
Quasicrystals and Poisson's summation formula.
Invent. Math. \textbf{200} (2015), 585--606.

\bibitem[MM08]{matei-meyer-quasicrystals}
B. Matei, Y. Meyer,
Quasicrystals are sets of stable sampling.
C. R. Acad. Sci. Paris, Ser. I \textbf{346} (2008), 1235--1238.

\bibitem[MM10]{matei-meyer-simple}
B. Matei, Y. Meyer,
Simple quasicrystals are sets of stable sampling.
Complex Var. Elliptic Equ. \textbf{55} (2010), 947--964.

\bibitem[Mey70]{mey0}
Y. Meyer, 
Nombres de Pisot, nombres de Salem et analyse harmonique.
Lecture Notes in Mathematics \textbf{117}, Springer-Verlag, 1970.

\bibitem[Mey95]{mey2}
Y. Meyer,
Quasicrystals, diophantine approximation and algebraic numbers.
Beyond quasicrystals (Les Houches, 1994), 3--16, Springer, Berlin, 1995.

\bibitem[OU06]{ou1}
A.~Olevskii and A.~Ulanovskii, Universal sampling of band-limited signals.
C. R. Math. Acad. Sci. Paris \textbf{342} (2006), 927--931. 

\bibitem[OU08]{ou2}
A. Olevskii, A. Ulanovskii,
Universal sampling and interpolation of band-limited signals.  
Geom. Funct. Anal. \textbf{18} (2008), 1029--1052. 

\end{thebibliography}
\end{document}